\documentclass{amsart}
\usepackage{amssymb}
\usepackage{amsbsy, amsthm, amsmath,amstext, amsopn}
\usepackage[all]{xy}
\usepackage{amsfonts}
\usepackage{amscd}
\usepackage{stmaryrd}
\usepackage{color}
\usepackage[numbers, square, sort, comma]{natbib}

\newtheorem{thm}{Theorem} [section]

\newtheorem{prop}[thm]{Proposition}

\theoremstyle{definition}

\newtheorem{defn}[thm]{Definition}

\newtheorem{example}[thm]{Example}

\theoremstyle{remark}

\newtheorem{remark}[thm]{Remark}

\begin{document}

\numberwithin{equation}{section}

\newcommand{\hs}{\mbox{\hspace{.4em}}}
\newcommand{\ds}{\displaystyle}
\newcommand{\bd}{\begin{displaymath}}
\newcommand{\ed}{\end{displaymath}}
\newcommand{\bcd}{\begin{CD}}
\newcommand{\ecd}{\end{CD}}

\newcommand{\on}{\operatorname}
\newcommand{\proj}{\operatorname{Proj}}
\newcommand{\bproj}{\underline{\operatorname{Proj}}}

\newcommand{\spec}{\operatorname{Spec}}
\newcommand{\Spec}{\operatorname{Spec}}
\newcommand{\bspec}{\underline{\operatorname{Spec}}}
\newcommand{\pline}{{\mathbf P} ^1}
\newcommand{\aline}{{\mathbf A} ^1}
\newcommand{\pplane}{{\mathbf P}^2}
\newcommand{\aplane}{{\mathbf A}^2}
\newcommand{\coker}{{\operatorname{coker}}}
\newcommand{\ldb}{[[}
\newcommand{\rdb}{]]}

\newcommand{\Sym}{\operatorname{Sym}}
\newcommand{\Symp}{\operatorname{Sym}}
\newcommand{\Pic}{\bf{Pic}}
\newcommand{\Aut}{\operatorname{Aut}}
\newcommand{\PAut}{\operatorname{PAut}}

\newcommand{\too}{\twoheadrightarrow}
\newcommand{\C}{{\mathbf C}}
\newcommand{\Z}{{\mathbf Z}}
\newcommand{\Q}{{\mathbf Q}}
\newcommand{\Cx}{{\mathbf C}^{\times}}
\newcommand{\Cbar}{\overline{\C}}
\newcommand{\Cxbar}{\overline{\Cx}}
\newcommand{\cA}{{\mathcal A}}
\newcommand{\cS}{{\mathcal S}}
\newcommand{\cV}{{\mathcal V}}
\newcommand{\cM}{{\mathcal M}}
\newcommand{\bA}{{\mathbf A}}
\newcommand{\cB}{{\mathcal B}}
\newcommand{\cC}{{\mathcal C}}
\newcommand{\cD}{{\mathcal D}}
\newcommand{\D}{{\mathcal D}}
\newcommand{\cs}{{\mathbf C} ^*}
\newcommand{\boldc}{{\mathbf C}}
\newcommand{\cE}{{\mathcal E}}
\newcommand{\cF}{{\mathcal F}}
\newcommand{\bF}{{\mathbf F}}
\newcommand{\cG}{{\mathcal G}}
\newcommand{\G}{{\mathbb G}}
\newcommand{\cH}{{\mathcal H}}
\newcommand{\CI}{{\mathcal I}}
\newcommand{\cJ}{{\mathcal J}}
\newcommand{\cK}{{\mathcal K}}
\newcommand{\cL}{{\mathcal L}}
\newcommand{\baL}{{\overline{\mathcal L}}}

\newcommand{\Mf}{{\mathfrak M}}
\newcommand{\bM}{{\mathbf M}}
\newcommand{\bm}{{\mathbf m}}
\newcommand{\cN}{{\mathcal N}}
\newcommand{\theo}{\mathcal{O}}
\newcommand{\cP}{{\mathcal P}}
\newcommand{\cR}{{\mathcal R}}
\newcommand{\Pp}{{\mathbb P}}
\newcommand{\boldp}{{\mathbf P}}
\newcommand{\boldq}{{\mathbf Q}}
\newcommand{\bbL}{{\mathbf L}}
\newcommand{\cQ}{{\mathcal Q}}
\newcommand{\cO}{{\mathcal O}}
\newcommand{\Oo}{{\mathcal O}}
\newcommand{\cY}{{\mathcal Y}}
\newcommand{\OX}{{\Oo_X}}
\newcommand{\OY}{{\Oo_Y}}
\newcommand{\otY}{{\underset{\OY}{\ot}}}
\newcommand{\otX}{{\underset{\OX}{\ot}}}
\newcommand{\cU}{{\mathcal U}}\newcommand{\cX}{{\mathcal X}}
\newcommand{\cW}{{\mathcal W}}
\newcommand{\boldz}{{\mathbf Z}}
\newcommand{\qgr}{\operatorname{q-gr}}
\newcommand{\gr}{\operatorname{gr}}
\newcommand{\rk}{\operatorname{rk}}
\newcommand{\SH}{{\underline{\operatorname{Sh}}}}
\newcommand{\End}{\operatorname{End}}
\newcommand{\uEnd}{\underline{\operatorname{End}}}
\newcommand{\Hom}{\operatorname{Hom}}
\newcommand{\uHom}{\underline{\operatorname{Hom}}}
\newcommand{\uHomY}{\uHom_{\OY}}
\newcommand{\uHomX}{\uHom_{\OX}}
\newcommand{\Ext}{\operatorname{Ext}}
\newcommand{\bExt}{\operatorname{\bf{Ext}}}
\newcommand{\Tor}{\operatorname{Tor}}

\newcommand{\inv}{^{-1}}
\newcommand{\airtilde}{\widetilde{\hspace{.5em}}}
\newcommand{\airhat}{\widehat{\hspace{.5em}}}
\newcommand{\nt}{^{\circ}}
\newcommand{\del}{\partial}

\newcommand{\supp}{\operatorname{supp}}
\newcommand{\GK}{\operatorname{GK-dim}}
\newcommand{\hd}{\operatorname{hd}}
\newcommand{\id}{\operatorname{id}}
\newcommand{\res}{\operatorname{res}}
\newcommand{\lrar}{\leadsto}
\newcommand{\im}{\operatorname{Im}}
\newcommand{\HH}{\operatorname{H}}
\newcommand{\TF}{\operatorname{TF}}
\newcommand{\Bun}{\operatorname{Bun}}

\newcommand{\F}{\mathcal{F}}
\newcommand{\Ff}{\mathbb{F}}
\newcommand{\nthord}{^{(n)}}
\newcommand{\Gr}{{\mathfrak{Gr}}}

\newcommand{\Fr}{\operatorname{Fr}}
\newcommand{\GL}{\operatorname{GL}}
\newcommand{\gl}{\mathfrak{gl}}
\newcommand{\SL}{\operatorname{SL}}
\newcommand{\ff}{\footnote}
\newcommand{\ot}{\otimes}
\def\Ext{\operatorname {Ext}}
\def\Hom{\operatorname {Hom}}
\def\Ind{\operatorname {Ind}}
\def\bbZ{{\mathbb Z}}

\newcommand{\nc}{\newcommand}
\nc{\ol}{\overline} \nc{\cont}{\on{cont}} \nc{\rmod}{\on{mod}}
\nc{\Mtil}{\widetilde{M}} \nc{\wb}{\overline} 
\nc{\wh}{\widehat} \nc{\sm}{\setminus} \nc{\mc}{\mathcal}
\nc{\mbb}{\mathbb}  \nc{\K}{{\mc K}} \nc{\Kx}{{\mc K}^{\times}}
\nc{\Ox}{{\mc O}^{\times}} \nc{\unit}{{\bf \on{unit}}}
\nc{\boxt}{\boxtimes} \nc{\xarr}{\stackrel{\rightarrow}{x}}

\nc{\Ga}{\G_a}
 \nc{\PGL}{{\on{PGL}}}
 \nc{\PU}{{\on{PU}}}

\nc{\h}{{\mathfrak h}} \nc{\kk}{{\mathfrak k}}
 \nc{\Gm}{\G_m}
\nc{\Gabar}{\wb{\G}_a} \nc{\Gmbar}{\wb{\G}_m} \nc{\Gv}{G^\vee}
\nc{\Tv}{T^\vee} \nc{\Bv}{B^\vee} \nc{\g}{{\mathfrak g}}
\nc{\gv}{{\mathfrak g}^\vee} \nc{\BRGv}{\on{Rep}\Gv}
\nc{\BRTv}{\on{Rep}T^\vee}
 \nc{\Flv}{{\mathcal B}^\vee}
 \nc{\TFlv}{T^*\Flv}
 \nc{\Fl}{{\mathfrak Fl}}
\nc{\BRR}{{\mathcal R}} \nc{\Nv}{{\mathcal{N}}^\vee}
\nc{\St}{{\mathcal St}} \nc{\ST}{{\underline{\mathcal St}}}
\nc{\Hec}{{\bf{\mathcal H}}} \nc{\Hecblock}{{\bf{\mathcal
H_{\alpha,\beta}}}} \nc{\dualHec}{{\bf{\mathcal H^\vee}}}
\nc{\dualHecblock}{{\bf{\mathcal H^\vee_{\alpha,\beta}}}}
\newcommand{\ramBun}{{\bf{Bun}}}
\newcommand{\ramBuno}{\ramBun^{\circ}}

\nc{\Buntheta}{{\bf Bun}_{\theta}} \nc{\Bunthetao}{{\bf
Bun}_{\theta}^{\circ}} \nc{\BunGR}{{\bf Bun}_{G_\BR}}
\nc{\BunGRo}{{\bf Bun}_{G_\BR}^{\circ}}
\nc{\HC}{{\mathcal{HC}}}
\nc{\risom}{\stackrel{\sim}{\to}} \nc{\Hv}{{H^\vee}}
\nc{\bS}{{\mathbf S}}
\def\BRep{\operatorname {Rep}}
\def\Conn{\operatorname {Conn}}

\nc{\Vect}{{\operatorname{Vect}}}
\nc{\Hecke}{{\operatorname{Hecke}}}

\newcommand{\ZZ}{{Z_{\bullet}}}
\nc{\HZ}{{\mc H}\ZZ} \nc{\eps}{\epsilon}

\nc{\CN}{\mathcal N} \nc{\BA}{\mathbb A}

\nc{\ul}{\underline}

\nc{\bn}{\mathbf n} \nc{\Sets}{{\on{Sets}}} \nc{\Top}{{\on{Top}}}
\nc{\IntHom}{{\mathcal Hom}}

\nc{\Simp}{{\mathbf \Delta}} \nc{\Simpop}{{\mathbf\Delta^\circ}}

\nc{\Cyc}{{\mathbf \Lambda}} \nc{\Cycop}{{\mathbf\Lambda^\circ}}

\nc{\Mon}{{\mathbf \Lambda^{mon}}}
\nc{\Monop}{{(\mathbf\Lambda^{mon})\circ}}

\nc{\Aff}{{\on{Aff}}} \nc{\Sch}{{\on{Sch}}}

\nc{\bul}{\bullet}
\nc{\module}{{\operatorname{-mod}}}

\nc{\dstack}{{\mathcal D}}

\nc{\BL}{{\mathbb L}}

\nc{\BD}{{\mathbb D}}

\nc{\BR}{{\mathbb R}}

\nc{\BT}{{\mathbb T}}

\nc{\SCA}{{\mc{SCA}}}
\nc{\DGA}{{\mc DGA}}

\nc{\DSt}{{DSt}}

\nc{\lotimes}{{\otimes}^{\mathbf L}}

\nc{\bs}{\backslash}

\nc{\Lhat}{\widehat{\mc L}}

\newcommand{\Coh}{\on{Coh}}

\nc{\QCoh}{QC}
\nc{\QC}{QC}
\nc{\Perf}{\on{Perf}}
\nc{\Cat}{{\on{Cat}}}
\nc{\dgCat}{{\on{dgCat}}}
\nc{\bLa}{{\mathbf \Lambda}}

\nc{\BRHom}{\mathbf{R}\hspace{-0.15em}\on{Hom}}
\nc{\BREnd}{\mathbf{R}\hspace{-0.15em}\on{End}}
\nc{\colim}{\on{colim}}
\nc{\oo}{\infty}
\nc{\Mod}{\on{Mod} }

\nc\fh{\mathfrak h}
\nc\al{\alpha}
\nc\la{\alpha}
\nc\BGB{B\bs G/B}
\nc\QCb{QC^\flat}
\nc\qc{\on{QC}}

\def\w{\wedge}
\nc{\vareps}{\varepsilon}

\nc{\fg}{\mathfrak g}

\nc{\Map}{\on{Map}} \nc{\fX}{\mathfrak X}

\nc{\ch}{\check}
\nc{\fb}{\mathfrak b} \nc{\fu}{\mathfrak u} \nc{\st}{{st}}
\nc{\fU}{\mathfrak U}
\nc{\fZ}{\mathfrak Z}

 \nc\fc{\mathfrak c}
 \nc\fs{\mathfrak s}

\nc\fk{\mathfrak k} \nc\fp{\mathfrak p}

\nc{\BRP}{\mathbf{RP}} \nc{\rigid}{\text{rigid}}
\nc{\glob}{\text{glob}}

\nc{\cI}{\mathcal I}

\nc{\La}{\mathcal L}

\nc{\quot}{/\hspace{-.25em}/}

\nc\aff{\it{aff}}
\nc\BS{\mathbb S}

\nc\Loc{{\mc Loc}}
\nc\Tr{{\on{Tr}}}
\nc\Ch{{\mc Ch}}

\nc\ftr{{\mathfrak {tr}}}
\nc\fM{\mathfrak M}

\nc\Id{\operatorname{Id}}

\nc\bimod{\on{-bimod}}

\nc\ev{\operatorname{ev}}
\nc\coev{\operatorname{coev}}

\nc\pair{\operatorname{pair}}
\nc\kernel{\operatorname{kernel}}

\nc\Alg{\operatorname{Alg}}

\nc\init{\emptyset_{\text{\em init}}}
\nc\term{\emptyset_{\text{\em term}}}

\nc\Ev{\on{Ev}}
\nc\Coev{\on{Coev}}

\nc\es{\emptyset}
\nc\m{\text{\it min}}
\nc\M{\text{\it max}}
\nc\cross{\text{\it cr}}
\nc\tr{\on{tr}}

\nc\perf{\on{-perf}}
\nc\inthom{\mathcal Hom}
\nc\intend{\mathcal End}

\newcommand{\Sh}{\on{Sh}}

\nc{\Comod}{\on{Comod}}
\nc{\cZ}{\mathcal Z}

\def\interiorsymbol {\on{int}}

\nc\frakf{\mathfrak f}
\nc\fraki{\mathfrak i}
\nc\frakj{\mathfrak j}
\nc\bP{\mathbb P}
\nc\stab{st}
\nc\Stab{St}

\nc\fN{\mathfrak N}
\nc\fT{\mathfrak T}
\nc\fV{\mathfrak V}

\nc\Ob{\on{Ob}}

\nc\fC{\mathfrak C}
\nc\Fun{\on{Fun}}

\nc\Null{\on{Null}}

\nc\BC{\mathbb C}

\nc\loc{\on{Loc}}

\nc\hra{\hookrightarrow}
\nc\fL{\mathfrak L}
\nc\R{\mathbb R}
\nc\CE{\mathcal E}

\nc\sK{\mathsf K}
\nc\sC{\mathsf C}

\nc\Cone{\mathit Cone}

\nc\fY{\mathfrak Y}
\nc\fe{\mathfrak e}
\nc\ft{\mathfrak t}

\nc\wt{\widetilde}
\nc\inj{\mathit{inj}}
\nc\surj{\mathit{surj}}

\nc\Path{\mathit{Path}}
\nc\Set{\mathit{Set}}
\nc\Fin{\mathit{Fin}}

\nc\cyc{\mathit{cyc}}

\nc\per{\mathit{per}}

\nc\sym{\mathit{symp}}
\nc\con{\mathit{cont}}
\nc\gen{\mathit{gen}}
\nc\str{\mathit{str}}
\nc\rsdl{\mathit{res}}
\nc\rel{\mathit{rel}}
\nc\pt{\mathit{pt}}
\nc\naive{\mathit{nv}}
\nc\forget{\mathit{For}}
\nc\periodic{\mathit{Per}}

\nc\sW{\mathsf W}
\nc\sE{\mathsf E}
\nc\sP{\mathsf P}
\nc\sB{\mathsf B}
\nc\sS{\mathsf S}
\nc\fH{\mathfrak H}
\nc\fP{\mathfrak P}
\nc\fW{\mathfrak W}
\nc\fE{\mathfrak E}
\nc\sx{\mathsf x}
\nc\sy{\mathsf y}

\nc\ord{\mathit{ord}}

\title[Cyclic symmetries of  $A_n$-quiver representations]{Cyclic symmetries  of $A_n$-quiver representations}

\author{David Nadler}
\address{Department of Mathematics\\University of California, Berkeley\\Berkeley, CA  94720-3840}
\email{nadler@math.berkeley.edu}

%

\maketitle


 \tableofcontents


\section{Introduction}



This short note contains a combinatorial construction of symmetries  arising in several distinct but related areas. 
In symplectic geometry, it is connected to categorical quantizations in the form of microlocal sheaves~\cite{KS} or partially wrapped or infinitesimal Fukaya categories~\cite{Seidel, AurouxICM},  in particular of Riemann surfaces in the form of ribbon graphs~\cite{harer, penner, kont1, kont2} as  explored in~\cite{STZ}. In algebraic geometry, via mirror symmetry, it is connected to derived categories of singularities~\cite{eis, Orlov}, in particular  of functions on curves and more general Brieskorn singularities
via duality and the Thom-Sebastiani theorem~\cite{KR, dyck, preygel}. In $K$-theory, it is connected to Waldhausen's S-construction~\cite{Waldhausen},
specifically with its corepresentability and $S^1$-equivariance~\cite{DK}.\footnote{I  am grateful to J.~Lurie for discussions about how this note  contributes to an understanding of the $S^1$-equivariance of the S-construction as appears in the comprehensive work of Dyckerhoff-Kapranov (see the discussion of~\cite[Introduction, p. 9]{DK}).}
Our primary motivation (in the spirit of~\cite{kont}, and to be taken up in general elsewhere~\cite{deform}) lies in the first direction, with the aim of  constructing a combinatorial  quantization of Lagrangian skeleta, equivalent to microlocal sheaves in their many guises. We explain below the one dimensional case of ribbon graphs where the main result  of this note gives an immediate solution.  
 
 To state the main result of this note (in its simplest two-periodic form), we briefly recall Connes' cyclic category~\cite{Connes, DHK, Loday}. Let $\ul \Lambda$ denote the big cyclic category of  finite cyclically ordered   nonempty sets.
Objects are finite nonempty subsets of the circle $S\subset S^1$, and morphisms $S\to S'$ are homotopy classes of degree $1$ maps $\varphi: S^1\to S^1$ such that $\varphi(S) \subset S'$.
The traditional cyclic category $\Lambda$ is the full subcategory  of $ \ul\Lambda$ on the objects $[n] = \Z/(n+1)\Z \subset S^1$ embedded as $(n+1)$st roots of unity,
for $n=0, 1, 2, \ldots$.  The inclusion $\Lambda \subset \ul\Lambda$ is  an equivalence.

For each finite cyclically ordered   nonempty set  $S$, we introduce a triangulated two-periodic $A_\oo$-category $\cC_{S, st}$. (It is the twisted complexes in a particularly simple  two-periodic  $A_\oo$-category  $\cC_{S}$ with objects given by the set $S$.) It provides a symmetric presentation of the  two-periodic differential graded derived category of finite-dimensional representations of the $A_{n-1}$-quiver
with $n=|S|$. For example, for $n=0$, it is the zero category, but presented as one object $s_0$ with a degree one map
$$
\xymatrix{
s_0 \ar[r] & s_0
}
$$ 
whose differential is the (two-periodic) identity.
For $n=1$, it is generated by two objects $s_0, s_1$ with degree one maps
$$
\xymatrix{
s_0 \ar[r] & s_1 &  s_1\ar[r] &  s_0
}
$$ whose compositions are the respective (two-periodic) identities. For $n=2$, it provides a symmetric presentation of a universal exact triangle (as explained in~\cite{Seidel}, and attributed to Kontsevich).  For general $n$,  it provides a symmetric presentation of $n$-step filtrations,
which
 allows for an easy verification of  the following. (From a concrete perspective, the primary content is in the precise
 combinatorial form of the construction rather than the abstract statement itself.)

\begin{thm}[Theorem~\ref{main thm} below]\label{intro main thm}
The assignment of the two-periodic $A_\oo$-category $\cC_{S,st}$ to a finite cyclically ordered nonempty set $S$ naturally extends to a functor 
from the opposite of the cyclic category to the strict category of two-periodic $A_\oo$-categories
$$
\xymatrix{
\cC_{st}:\ul \Lambda^{op}\ar[r] &  \cA_\per^\str
}
$$
\end{thm}

\begin{remark}
We also describe a graded version in Theorem~\ref{main thm gr} where the target is the category $\cA_\oo$ of plain (no longer two-periodic) $A_\oo$-categories. Instead of the cyclic category, the domain becomes a cover of the cyclic category with 
objects comprising pairs of a finite nonempty subset $S\subset S^1$ and a point $c\in \Sym^2(S^1)$. 
\end{remark}

\begin{remark}
Thanks to the natural duality equivalence $\ul\Lambda \simeq \ul \Lambda^{op}$, the theorem also provides an alternatively variant functor.
\end{remark}

Because of its basic nature as a universal sequence of composable morphisms, the $A_n$-quiver appears wherever categories appear. Similarly,  its derived category of representations appears wherever triangulated categories appear. What is less immediately evident in some contexts (though certainly of primary focus in others) is the importance of the natural  functors (as appear in Theorem~\ref{intro main thm}) between representations of the $A_n$-quivers  for varying $n$. Let us informally mention three settings where they play a significant role.

\subsection{Symplectic geometry}
 Representations of the $A_{n-1}$-quiver arise in symplectic geometry as the infinitesimal (or dual partially wrapped) Fukaya category of the exact symplectic manifold $M = \BC$ with support Lagrangian $\Lambda_n\subset \BC$ the union of $n$  rays emanating from the origin. 
The functoriality of Theorem~\ref{intro main thm} provides the gluing one needs to extend this description to 
 the general
 case of oriented surfaces as captured by ribbon graphs~\cite{harer, penner, kont1, kont2} as explored in~\cite{STZ}. 

To spell this out, by a ribbon graph we will mean a locally finite graph $\Gamma$ with an embedding into the germ of an oriented surface $\Sigma$. We can allow  $\Gamma$ to have half-infinite edges incident to a single vertex, as well as infinite or circular edges incident to no vertices. We can also allow $\Gamma$ to have vertices of any finite nonzero valency.

There is a natural cosheaf of sets over $\Gamma$ given by the local components of the complement $\Sigma\setminus \Gamma$.
Moreover, the cosheaf naturally takes values in the cyclic category $\ul \Lambda$ thanks to the orientation of $\Sigma$. Composing with the functor of Theorem~\ref{intro main thm}, we obtain a sheaf of $A_\oo$-categories  (or alternatively, 
via  the natural  equivalence $\ul\Lambda \simeq \ul \Lambda^{op}$,  a cosheaf of $A_\oo$-categories) over $\Gamma$.
Taking its global sections in the form of a homotopy limit (or alternatively, homotopy colimit) over $\Gamma$, we obtain a combinatorial model for the two-periodic
infinitesimal (or alternatively, partially wrapped) Fukaya category of  $\Sigma$. (Note that one could go further and add a bicanonical trivialization and then apply the graded functor of Theorem~\ref{main thm gr}.)
This is the  first instance 
of general constructions to be further developed in~\cite{deform}. 

%

%

\subsection{Algebraic geometry} 
Under mirror symmetry, the above symplectic realization of representations of the $A_{n-1}$-quiver arises in algebraic geometry as
$\Gm$-equivariant matrix factorizations for the singularity $x^{n} = 0$. Here 
the functoriality of Theorem~\ref{intro main thm} is less evident and could be profitably organized in terms of integral transforms~\cite{KR, dyck, preygel}
living on Brieskorn singularities. (In informal discussions, J. Lurie has also explained a  simple formulation of it
in a flexible homotopical setting).

\subsection{$K$-theory} Given a suitable stable category, Waldhausen's S-construction records the structure of filtered objects and their  relationships under restriction and induction of filtrations. By keeping track of the natural symmetries of such operations, 
 it produces a paracyclic space, and in particular,  from a two-periodic category, a cyclic space~\cite{DK}.
 Within the $A_\oo$-setting, one can interpret a graded version  of Theorem~\ref{intro main thm} as the calculation of the object corepresenting the S-construction.
In fact, the graded version we present in Theorem~\ref{main thm gr} records slightly more symmetry than the S-construction:
it only depends on a point $c\in \Sym^2(S^1)$ rather than a single point $c\in S^1$ as would appear in a paracyclic realization.

\subsection{Acknowledgements}
It is a pleasure to thank Jacob Lurie for his interest in the construction recorded here and stimulating discussions of generalizations to an $\oo$-categorical setting. I am  grateful to David Ben-Zvi for  many beautiful explanations of related topics,
and to Toby Dyckerhoff for generous explanations of his joint work with Mikhail Kapranov. I am also grateful to Denis Auroux, David Treumann, and Eric Zaslow for their feedback on the broader undertaking of which this is a part.

I gratefully acknowledge the support of NSF grant DMS-1319287.


\section{Two-periodic $A_\oo$-categories}

Fix a base field $k$. Introduce the commutative  graded $k$-algebra $k[u, u^{-1}]$ where $\deg u = 2$.
All of our  constructions will be {two-periodic} in the sense that they will be $k[u, u^{-1}]$-linear.

\begin{remark}
Our constructions make sense with $k$ replaced by the integers or  any base commutative ring. Moreover, J.~Lurie has sketched an $\oo$-categorical version that takes place over the sphere spectrum.
\end{remark}

We will follow the conventions on $A_\oo$-categories from Seidel's book~\cite{Seidel}.
By a two-periodic $A_\oo$-category $\cC$, we will mean the following data:
\begin{enumerate}
\item A set of objects $\Ob \cC$.
\item For pairs of objects $c_0, c_1\in \Ob \cC$,  a compatibly graded $k[u, u^{-1}]$-module $\hom_\cC(c_0, c_1)$.
\item For every $d\geq 1$, and objects $c_0, \ldots, c_d\in \Ob \cC$, a $k[u, u^{-1}]$-linear composition map
$$
\xymatrix{
\mu_d:\hom_\cC(c_{d-1}, c_d) \otimes \cdots \otimes \hom_\cC(c_{0}, c_1) \ar[r] & \hom_\cC(c_0, c_d)[2-d]
}
$$
The composition maps must satisfy the quadratic $A_\oo$-equations
$$
\xymatrix{
\displaystyle
\sum_{m = 0}^d \sum_{n =0}^{d-m} (-1)^{\dag_n} \mu_{d-m+1} (a_d, \ldots, a_{n+m+1}, \mu_m(a_{n+m}, \ldots, a_{n+1}), a_n, \ldots, a_1) = 0
}
$$
where the sign is determined by $\dag_n = |a_1| + \cdots + |a_n| - n$.
\end{enumerate}

\begin{remark}
We could equivalently assume our morphisms form $\Z/2\Z$-graded $k$-modules. Then the composition  map $\mu_d$ would be a morphism of degree $2-d = d$ modulo $2$.
\end{remark}

\begin{remark}
All of our $A_\oo$-categories will be strictly unital: for each object $c \in \Ob \cC$, there is a degree zero element $\id_c\in \hom^0_\cC(c, c)$,
called the identity morphism, such that
\begin{enumerate}
\item for  any $a\in \hom_\cC(c_0, c_1)$, we have
$$
(-1)^{|a|}\mu_2(\id_{c_1}, a) = a = \mu_2(a, \id_{c_0}).
$$
 \item for  $d\not =2$, the composition map $\mu_d$ vanishes if any entry is an identity morphism.
\end{enumerate}
\end{remark}

\begin{defn}\label{defn per cat}
Let $S = [s_0, s_1, \ldots, s_n] \subset S^1$ be a  finite cyclically ordered nonempty set. We will understand the indices modulo $n+1$
so that $s_{i+n+1}$ also denotes $s_i$.

We define the unital $k[u, u^{-1}]$-linear $A_\oo$-category  $\cC_{S}$ as follows.

(1) Objects: $\Ob \cC_{S}= S$.

(2) Morphisms: the free $k[u, u^{-1}]$-modules generated by  the identity morphisms and  
  additional degree one morphisms
$$
\xymatrix{
 v_{i} \in \hom_{\cC_S}(s_i, s_{i+1}),
&
|v_{i} |= 1,
&
\mbox{for $i =0, \ldots, n$.}
}
$$

(3)  Compositions: all are  zero except for the $k[u, u^{-1}]$-linear maps given by the  identity compositions
and the additional compositions
$$
\xymatrix{
\mu_{n+1}( v_{i+ n}, \ldots, v_{i}) = u \id_{s_i},
&
\mbox{for $i =0, \ldots, n$.}
}
$$
\end{defn}

To check that $\cC_S$ satisfies the  quadratic $A_\oo$-equations, it suffices to consider composable sequences of the (non-identity) degree one generating  morphisms. By construction, each term of the $A_\oo$-equations  individually vanishes on such sequences, except for sequences of the form 
$$
\xymatrix{
s_i \ar[r]^-{v_i} &  s_{i+1} \ar[r]^{v_{i+1}} & \cdots  \ar[r]^-{v_{i+n}} & s_{i+n +1}  \ar[r]^-{v_{i}} & s_{i+1},
& \mbox{for $i =0, \ldots, n$.}
}
$$
For such  sequences, two terms will not vanish, but indeed cancel each other
$$\begin{array}{c}
 \mu_2(v_i, \mu_{n+1}(v_{i+n}, \ldots, v_{i+1}, v_i))
+ (-1)^{|v_i|-1}\mu_2(\mu_{n+1}(v_i, v_{i+n}, \ldots, v_{i+1}), v_i) \\
= \mu_2(v_i, u\id_{s_i})
+ \mu_2(u\id_{s_{i+1}}, v_i) 
= u (v_i + (-1)^{|v_i|} v_i) = 0
\end{array}
$$

\begin{remark}
Note that the  (non-identity)  degree one generating morphisms of $\cC_S$ are in natural bijection with the components of the complement $S^1\setminus S$.
\end{remark}

\begin{example}[$n=0$]\label{ex n=0}
When $S=[s_0]$, we see that $\cC_{S}$ consists of one object $s_0$ with  endomorphism complex
$$
\xymatrix{
\hom_{\cC_S}(s_0, s_0) = k[u, u^{-1}] \id_{s_0}   \oplus k[u, u^{-1}]v_{0}
&
\mu_1(v_{0}) = u \id_{s_0}
}$$

Thus $\hom_{\cC_S}(s_0, s_0)$ is acyclic, and hence $\cC_S$ is quasi-equivalent to the zero $A_\oo$-category.
\end{example}

\begin{example}[$n=1$]
When $S=[s_0, s_1]$, we see that $\cC_S$ consists of two objects $s_0, s_1$ with  morphism complexes
$$
\xymatrix{
\hom_{\cC_S}(s_0, s_0) = k[u, u^{-1}] \id_{s_0}
& 
\hom_{\cC_S}(s_1, s_1) = k[u, u^{-1}] \id_{s_1}
}$$
$$
\xymatrix{
\hom_{\cC_S}(s_0, s_1) = k[u, u^{-1}] v_0
& 
\hom_{\cC_S}(s_1, s_0) = k[u, u^{-1}] v_1
}$$
with the only non-zero or non-identity compositions  the $k[u, u^{-1}]$-linear maps given by
$$
\xymatrix{
\mu_2(v_1, v_0) = u \id_{s_0}
& 
\mu_2(v_0, v_1) = u \id_{s_1}}
$$

Thus $\cC_S$ is non-canonically equivalent to the full $A_\oo$-subcategory of graded $k[u, u^{-1}]$-modules given by the free module $ k[u, u^{-1}]$ and its odd shift $ k[u, u^{-1}] [1]$. To give such an equivalence, either object $s_0$ or $s_1$ can be sent to $ k[u, u^{-1}]$,
and then the other will be sent to $ k[u, u^{-1}][1]$.
 \end{example}

%


\section{Pointed version: injective cyclic functoriality}
We continue to follow the conventions on $A_\oo$-categories from Seidel's book~\cite{Seidel}.
By an $A_\oo$-functor $\cF:\cC\to \cD$ between two-periodic $A_\oo$-categories, we will mean the following data:
\begin{enumerate}
\item A map of sets $\cF:\Ob \cC\to\Ob \cD$.
\item For every $d\geq 1$, and objects $c_0, \ldots, c_d\in \Ob \cC$, a $k[u, u^{-1}]$-linear  map
$$
\xymatrix{
\cF_d:\hom_\cC(c_{d-1}, c_d) \otimes \cdots \otimes \hom_\cC(c_{0}, c_1) \ar[r] & \hom_\cD(\cF c_0, \cF c_d)[1-d]
}
$$
The  maps must satisfy the polynomial equations
$$
\begin{array}{c}
\displaystyle
\sum_r \sum_{k_1, \ldots, k_r}
\mu_r(\cF_{k_r}(a_{d}, \ldots, a_{d-k_r+1}), \ldots, \cF_{k_1}(a_{k_1}, \ldots, a_1))\\
=
\displaystyle
\sum_{m = 0}^d \sum_{n =0}^{d-m} (-1)^{\dag_n} \cF_{d-m+1} (a_d, \ldots, a_{n+m+1}, \mu_m(a_{n+m}, \ldots, a_{n+1}), a_n, \ldots, a_1) 
\end{array}
$$
where the sums of the left hand side run over all $r\geq 1$ and partitions $k_1 +\cdots + k_r = d$,
and the sign of the right hand side is determined by $\dag_n = |a_1| + \cdots + |a_n| - n$.
\end{enumerate}

The composition of $A_\oo$-functors $\cF:\cC\to \cD$, $\cG:\cD\to \cE$ is the $A_\oo$-functor defined by:
\begin{enumerate}
\item The composite map of sets $\cG\circ \cF:\Ob \cC\to\Ob \cE$.
\item The composite $k[u, u^{-1}]$-linear maps 
$$
(\cG\circ \cF)_d(a_d, \ldots, a_1)= 
\displaystyle
\sum_r \sum_{k_1, \ldots, k_r}
\cG_r(\cF_{k_r}(a_{d}, \ldots, a_{d-k_r+1}), \ldots, \cF_{k_1}(a_{k_1}, \ldots, a_1))
$$
where the sums run over all $r\geq 1$ and partitions $k_1 +\cdots + k_r = d$,
\end{enumerate}

Composition is strictly associative with unit the identity functor $\id_\cC:\cC\to \cC$ given by the identity on objects, $(\id_\cC)_1$ the identity on morphisms, and $(\id_\cC)_d = 0$, for all $d\geq 2$.

\begin{remark}
All of our $A_\oo$-functors will be strictly unital: 
\begin{enumerate}
\item for each object $c \in \Ob \cC$, we have 
 $\cF_1(\id_c) = \id_{\cF(c)}$,
\item for $d\geq 2$, the map $\cF_d$ vanishes if any entry is an identity morphism.
\end{enumerate}

\end{remark}

\begin{defn}
Let $\cA^{\str}_\per$ denote the strict category with objects  $k[u,u^{-1}]$-linear  $A_\oo$-categories and morphisms 
$k[u,u^{-1}]$-linear   $A_\oo$-functors.
(We refer to it as strict since we will not consider any natural transformations between functors which are not the identity transformation.)
\end{defn}

Let $\ul \Lambda_{\inj} \subset \ul \Lambda$ denote the non-full subcategory with objects  finite cyclically ordered nonempty  sets and morphisms   cyclic maps that are injective as set maps. 

\begin{remark}
It is worth noting that the forgetful functor $\ul \Lambda\to \ul \Fin$ is not faithful, but its restriction
$\ul \Lambda_\inj \to \ul \Fin$ is in fact faithful.
\end{remark}

\begin{defn}
Let $S \subset S^1$ be a finite cyclically ordered nonempty   set.

We define the unital $k[u, u^{-1}]$-linear $A_\oo$-category  $\cC_{S, 0}$ to be the $A_\oo$-category $\cC_S$ with a zero object $0 
$ adjoined. In other words, $\Ob \cC_{S, 0} = \Ob \cC_S \sqcup \{0\}$, and the morphism spaces in $\cC_{S, 0}$ with domain or target $0$ are all zero.
\end{defn}

\begin{prop}\label{prop inj funct}
The assignment of the $A_\oo$-category $\cC_{S,0}$ to a finite cyclically ordered nonempty  set $S$ naturally extends to a functor 
on the injective cyclic category
$$
\xymatrix{
\ul \Lambda_{\inj}^{op}\ar[r] &  \cA_\per^\str
}
$$
\end{prop}

\begin{proof}
Let $f:S \to T$ be a cyclic injection.
  We will   define a unital $A_\oo$-functor 
$$
\xymatrix{
f^*:\cC_{T, 0}\ar[r] & \cC_{S, 0}
}
$$

For the map of objects $f^*:T\sqcup \{0\} \to S\sqcup\{0\}$, we set $f^*(0) = 0$, and for $t\in T$, we set
$$
f^*(t) 
= \left\{
\begin{array}{ll}
s & \mbox{ when $t= f(s)$}\\
 0 &  \mbox{ when $t\not\in  f(S)$}
\end{array}
\right.
$$

For the maps of morphisms, let us introduce the term {\em $f$-interval of length $k$},  to refer to any interval of elements $[t_0, \ldots, t_{k}] \subset T$ such that $t_0, t_{k}\in f(S)$ and $t_1, \ldots, t_{k-1}\not\in f(S)$, for some $k\geq 1$. (We allow the possibility that $t_0 = t_k$.) 

For $k\geq 1$, we define the map $f^*_k$ on composable sequences of the (non-identity) degree one  generating morphisms 
$$
\xymatrix{
t_0 \ar[r]^-{v_{0}} & t_1 \ar[r]^-{v_{1}} & \cdots \ar[r]^-{v_{k-1}} & t_{k}
}
$$
to be zero, except when $[t_0, \ldots, t_{k}]$ is an $f$-interval of length $k$, in which case we set
$$
f^*_k(v_{k-1}, \ldots, v_{0}) = w_0
$$
where  $w_0 \in \hom_{\cC_S}(s_0, s_1)$ denotes the degree one generating  morphism
for $s_0, s_1\in S$  the unique elements with $f(s_0) = t_0, f(s_1) = t_{k}$.

To check that $f^*$ defines a unital $A_\oo$-functor, it suffices to consider composable sequences of  the  (non-identity) degree one generating 
morphisms. 
By construction, each term of the equations individually vanishes on such sequences, except possibly for  complete cycles
$$
\xymatrix{
t_0 \ar[r]^-{v_{0}} & t_1 \ar[r]^-{v_{1}} & \cdots \ar[r]^-{v_{n-1}} & t_n  \ar[r]^-{v_{n}} & t_{0}
}
$$
Furthermore, for some term not to vanish, there must be a (necessarily unique) element $s_0\in S$ with $f(s_0) = t_0$.
For such  sequences, two terms will not vanish, but indeed equal each other
$$\begin{array}{c}
\mu_r(f^*_{k_r}(v_{n}, \ldots, v_{n- k_r+1}), \ldots, f^*_{k_1}(v_{k_1}, \ldots, v_0)) = u \id_{s_0} =   f^*_1(\mu_{n+1}(v_n,\ldots, v_0))
\end{array}
$$
where 
 the partition of the left hand side is the unique partition of the sequence into $f$-intervals.

Finally, we check that such functors compose as asserted.
Let $f:S \to T$, $g:T\to U$ be cyclic injections.  On objects, we clearly have $ (g\circ f)^* = f^*\circ g^* $.
On morphisms, it suffices to consider composable sequences of  the  (non-identity) degree one generating 
morphisms. 
By construction, each term of the equations individually vanishes on such sequences, except  for $(g\circ f)$-intervals
$$
\xymatrix{
u_0 \ar[r]^-{v_{0}} & u_1 \ar[r]^-{v_{1}} & \cdots \ar[r]^-{v_{k-1}} & u_k
}
$$
For such  sequences, two terms will not vanish, but indeed equal each other
$$\begin{array}{c}
 (g\circ f)^*_k(v_{k-1}, \ldots, v_0) = w_0 = g^*_r(f_{k_r}(v_k , \ldots, v_{k - k_r + 1}), \ldots, f_{k_1}(v_{k_1}, \ldots, v_0))
 \end{array}
$$
where  $w_0 \in \hom_{\cC_S}(s_0, s_1)$ denotes the degree one generating  morphism
for $s_0, s_1\in S$  the unique elements with $f(s_0) = t_0, f(s_1) = t_{k}$,
and
the partition of the right hand side is the unique partition of the sequence into $g$-intervals.
\end{proof}

\begin{remark}
It is worth informally noting that the functors of the proposition are the natural quotients by the objects not in the image of the cyclic set map.
\end{remark}


\section{Twisted complexes: surjective cyclic functoriality}

We continue to follow the conventions on $A_\oo$-categories from Seidel's book~\cite{Seidel}.

\begin{defn}
Let $S\subset S^1$ be a finite  cyclically ordered  nonempty set.

We define the unital $k[u, u^{-1}]$-linear $A_\oo$-category  $\cC_{S, st}$ to be the $A_\oo$-category of twisted complexes of $\cC_S$,
or equivalently of $\cC_{S, 0}$.
\end{defn}

Let $\cQ_n$ denote the  $k[u, u^{-1}]$-linear triangulated differential graded derived category of finite-dimensional modules over the $A_n$-quiver
$$
\xymatrix{
\underset{1}{\bullet} \ar@<+.1cm>[r] & \underset{2}{\bullet} \ar@<+.1cm>[r] & \cdots \ar@<+.1cm>[r] & \underset{n}{\bullet}
}
$$ 
For $i=1, \ldots, n$, denote by $m_i \in \cQ_n$  the free rank one module supported at the $i$th vertex and zero elsewhere.
Note that among these objects,  there are degree one generating morphisms $w_i:m_i \to m_{i+1}$, for $i=1, \ldots, n-1$, and no other linearly independent non-identity morphisms.

\begin{prop}\label{prop quiver}
Let $S= [s_0, s_1, \ldots, s_n]\subset S^1$ be a  finite cyclically ordered nonempty set.

Fix the element $s_0\in S$ so that the remaining elements inherit the linear order $s_1, \ldots, s_n$.

Then there is a quasi-equivalence of $A_\oo$-categories
$$
\xymatrix{
F:\cQ_n \ar[r]^-\sim & \cC_{S, st}
}
$$
such that $F(m_i) = s_i$, for $i=1, \ldots, n$, and $F(w_i) = v_i$, for $i=1, \ldots, n-1$.
%
\end{prop}

\begin{proof}
Clearly the functor $F$ is well-defined and a quasi-equivalence on the full subcategory of $\cC_{S, st}$ generated by $s_1, \ldots, s_n$. (Note there are no non-identity compositions among those objects.) To see it is essentially surjective, observe that $s_0$ is isomorphic to the shift by one of the twisted complex
$$
\xymatrix{
s_0' = (s_1\ar[r]^-{v_1} & s_2\ar[r]^-{v_2} & \cdots \ar[r]^-{v_{n-1}} & s_n)
}
$$
Namely the degree one morphisms  $v_0\in \hom_{\cC_S}(v_0, v_1)$ and $v_n\in \hom_{\cC_S}(v_n, v_0)$ induce degree zero  isomorphisms $v_0'\in \hom_{\cC_{S, st}}(s_0, s_0'[-1])$ and $v'_n\in \hom_{\cC_{S, st}}(s_0'[1], s_0)$ which are inverse (up to a two-periodic shift).
\end{proof}

Let $\ul \Lambda_{\surj} \subset \ul \Lambda$ denote the non-full subcategory with objects  finite cyclically ordered nonempty  sets and morphisms   cyclic maps that are surjective as set maps. 

\begin{remark}
It is worth noting that the forgetful functor $\ul \Lambda_\surj \to \ul \Fin$ is not faithful: for example, there are two distinct maps of cyclic objects  $\{1, -1\} \to \{1\}$
\end{remark}

\begin{prop}\label{prop inj funct}
The assignment of the $A_\oo$-category $\cC_{S,st}$ to a finite cyclically ordered nonempty  set $S$ naturally extends to a functor 
on the surjective cyclic category
$$
\xymatrix{
\ul \Lambda_{\surj}^{op}\ar[r] &  \cA_\per^\str
}
$$
\end{prop}

\begin{proof}
Let $f:S\to T$ be a cyclic surjection.   We will   define a unital $A_\oo$-functor 
$$
\xymatrix{
f^*:\cC_{T, st}\ar[r] & \cC_{S, st}
}
$$
For this, it suffices to define $f^*$ on the full $A_\oo$-subcategory $\cC_T \subset \cC_{T, st}$.

For  $t\in T$, consider the fiber $f^{-1}(t) = [s_0(t), \ldots, s_k(t)] \subset S$ as an interval equipped with its natural linear ordering.
 For the map of objects, we set $f^*(t)$ to be the twisted complex
$$
\xymatrix{
(s_{0}(t) \ar[r]^-{w_{0}} & s_{1}(t)  \ar[r]^-{w_{1}} & \cdots  \ar[r]^-{w_{k-1}} & s_{k}(t))
}$$
where $w_i\in \hom_{\cC_S}(s_i(t), s_{i+1}(t))$ is the degree one generating  morphism.
Note each morphism appearing has degree one and thus no shifts are needed.
 To confirm the twisted complex satisfies the generalized Maurer-Cartan equation,
 note that the fiber is an interval and so by construction, the composition maps on all subcollections of the morphisms appearing in the twisted complex vanish.

For the maps of morphisms, on the (non-identity) degree one generating morphisms $v_0\in \hom_{\cC_T}(t_0, t_1)$, 
we define  the first component $f^*_1(v_0) \in  \hom_{\cC_S}(f^*(t_0), f^*(t_1))$ to be the composition of three morphisms: projection to the last term
$$
\xymatrix{
(s_{0}(t_0) \ar[r]^-{w_{0,0}} & s_{1}(t_0)  \ar[r]^-{w_{0, 1}} & \cdots  \ar[r]^-{w_{0, k-1}} & s_{k}(t_0))
\ar[r] & s_{k}
}$$
followed by the  degree one generating morphism
$$
\xymatrix{
s_{k}(t_0) \ar[r]^{w_k} & s_{0}(t_1)
}
$$
followed by inclusion to the first term
$$
\xymatrix{
 s_{0}(t_1)\ar[r] &
(s_{0}(t_1) \ar[r]^-{w_{1, 0}} & \sigma_{1}(t_1)  \ar[r]^-{w_{1, 1}} & \cdots  \ar[r]^-{w_{1, \ell-1}} & s_{\ell}(t_1))
}
$$
Here $w_{0, i}\in \hom_{\cC_S}(s_i(t_0), s_{i+1}(t_0))$, $w_{1, j}\in \hom_{\cC_S}(s_j(t_1), s_{j+1}(t_1))$ denote the 
degree one
generating morphisms.

For $k>1$, we set the higher components $f^*_k$ to be identically zero.

To check that $f^*$ defines a unital $A_\oo$-functor, it suffices to consider composable sequences of  the  (non-identity) degree one generating 
morphisms. 
By construction, each term of the equations individually vanishes on such sequences, except possibly for  complete cycles
$$
\xymatrix{
t_0 \ar[r]^-{v_{0}} & t_1 \ar[r]^-{v_{1}} & \cdots \ar[r]^-{v_{n-1}} & t_n  \ar[r]^-{v_{n}} & t_{0}
}
$$
For such  sequences, two terms will not vanish, but indeed equal each other
$$\begin{array}{c}
\mu_{n+1}(f^*_{1}(v_{n}), \ldots, f^*_{1}(v_0)) = u \id_{s_0(t_0)} =   f^*_1(\mu_{n+1}(v_n,\ldots, v_0))
\end{array}
$$
where as above $s_0(t_0)\in S$ denotes the initial endpoint of the fiber  $f^{-1}(t_0)\subset S$.
Note that the evaluation of the left hand side follows from the constructions and the definition of component maps between twisted complexes
$$
\mu_{n+1}(f^*_{1}(v_{n}), \ldots, f^*_{1}(v_0)) =
\mu_{m+1}(w_m, \ldots, w_0)
$$
where the right hand side is evaluated on a complete cycle
$$
\xymatrix{
s_0 \ar[r]^-{w_{0}} & s_1 \ar[r]^-{w_{1}} & \cdots \ar[r]^-{w_{n-1}} & s_n  \ar[r]^-{w_{n}} & s_{0}
}
$$
where we start at $s_0 = s_0(t_0)\in S$

Finally, we check that such functors compose as asserted.
Let $f:S \to T$, $g:T\to U$ be cyclic surjections.  On objects, it is clear we  have $ (g\circ f)^* = f^*\circ g^* $.
On morphisms, it suffices to consider composable sequences of  the  (non-identity) degree one generating 
morphisms. 
By construction, each term of the equations individually vanishes on such sequences, except  for single terms
$v_0\in \hom_{\cC_U}(u_0, u_1)$, where it is clear we have $ (g\circ f)_1^* = f_1^*\circ g_1^*$.
\end{proof}

\begin{remark}
It is worth informally noting that the functors of the proposition are fully faithful.
\end{remark}

\section{Full cyclic functoriality}

\begin{thm}\label{main thm}
The assignment of the $A_\oo$-category $\cC_{S,st}$ to a finite cyclically ordered nonempty set $S$ naturally extends to a functor 
on the full cyclic category
$$
\xymatrix{
\cC_{st}:\ul \Lambda^{op}\ar[r] &  \cA_\per^\str
}
$$
\end{thm}

\begin{proof}
Given a cyclic map $f:S\to T$, consider the object $Q = f(S) \subset S^1$. We have a canonical cyclic surjection $f_\surj:S\to Q$ (given by $f$), 
and a canonical cyclic injection $f_\inj:Q\to T$ (given by the identity), along with  the evident equality $f=f_\inj \circ f_\surj$.

Let us define the  functor $f^*:\cC_{T, st} \to \cC_{S, st}$ to be the composition $f^* = f^*_\surj \circ f^*_\inj$
of our previously defined functors.  By our previous results, to see that this extends functorially, it suffices to show that if a cyclic map $f:S\to U$ happens to be a composition $f = g_\surj \circ h_\inj$ of an injection $h_\inj:S\to T$  followed by a surjection
 $g_\surj:T\to U$, then we have an equality
$$
\xymatrix{
f^* = f^*_\surj \circ f^*_\inj = h_\inj^* \circ g_\surj^*: \cC_{U, st} \ar[r] & \cC_{S, st}
}$$
Furthermore, it suffices to assume the sizes of the sets satisfy $|S| = |T| -1 = |U|$. It is convenient to consider two cases:
(1) the composition $f$ is an isomorphism; (2) the composition $f$ is not an isomorphism.

Case (1)  It suffices to consider $S=U = [s_0, \ldots, s_n]$, $T=[s_0, \ldots, s_n, t]$, with $h_\inj:S\to T$ the evident injection,
and $g_\surj:T\to U$ the  surjection such  that $g_\surj(t) = s_n$. (There is also the alternative case where $g_\surj(t) = s_0$, but it follows from a completely parallel argument.) 

We must check that $h_\inj^*\circ g_\surj^*:\cC_{S, st} \to \cC_{S, st}$ is the identity. On objects, this is clear, with the following observation: first forming the twisted complex $(s_n \to t)$ along the degree one generating morphism $ v_{n, t}\in \hom_{\cC_S}(s_n, t)$, then sending $t$ to the zero object results in the object $s_n$ again. 

On morphisms, it is also clear, with the following observation: by the definition of functors on twisted complexes,
  we have
$$
(h_\inj^*)_1(g_\surj^*)_1(v_n) = (h_\inj^*)_2(v_{t, 0}, v_{n, t}) = v_n \in \hom_{\cC_S}(s_n, s_0)
$$
for the degree one generating morphisms 
$$
v_n\in \hom_{\cC_S}(s_n, s_0)
\quad v_{n, t}\in \hom_{\cC_S}(s_n, t)
\quad v_{t, 0}\in \hom_{\cC_S}(t, s_0)
$$

Case (2) We will consider two subcases.

(a) The first subcase is particularly simple. Fix $1< i_0<n$, and consider $S= [s_0, \ldots, s_n]$, $T=[s_0, \ldots, s_n, t]$, $U= [s_0, \ldots, \hat s_{i_0}, \ldots, s_{n}, t]$, with $h_\inj:S\to T$ the evident injection,
and $g_\surj:T\to U$ the  surjection such  that $g_\surj(t) = t$, $g_\surj(s_j) = s_j$, for $j\not =  i_0$, and $g_\surj(s_{i_0}) = s_{i_0-1}$.
In this situation, the asserted identity is evident, since the injection and surjection do not interact with each other.

(b) Finally, it suffices to consider $S= [s_0, \ldots, s_n]$, $T=[s_0, \ldots, s_n, t]$, $U= [s_0, \ldots, s_{n-1}, t]$, with $h_\inj:S\to T$ the evident injection,
and $g_\surj:T\to U$ the  surjection such  that $g_\surj(s_n) = s_{n-1}$. (There is also the alternative case
where $S= [s_0, \ldots, s_n]$, $T=[s_0, \ldots, s_n, t]$, $U= [s_1, \ldots, s_n, t]$, with $h_\inj:S\to T$ the evident injection,
and $g_\surj:T\to U$ the  surjection such  that $g_\surj(s_0) = s_{1}$, but it follows from a completely parallel argument.)

We must check that $f^*_\surj \circ f_\inj^* = h_\inj^*\circ g_\surj^*:\cC_{U, st} \to \cC_{S, st}$ where $Q= f(S) = [s_0, \ldots, s_{n-1}]$
$f_\inj:Q\to U$ is the evident injection, and $f_\surj:S\to Q$ is the surjection such that $f(s_n) = s_{n-1}$.

On objects, this is immediate from the definitions: under both functors, the objects $s_0, \ldots, s_{n-2}$ are sent to themselves, $t$ is sent to $0$, and $s_{n-1}$ is sent to the  
 twisted complex $(s_{n-1} \to s_n)$ along the degree one generating morphism $ v_{n-1}\in \hom_{\cC_S}(s_{n-1}, s_n)$.

On morphisms, it is also clear, with the following observation: by the definition of functors on twisted complexes,
  we have
$$
f^*_2(v_{t, 0}, v_{n-1, t}) =  (h_\inj^*)_2((g_\surj^*)_1v_{t, 0}, (g_\surj^*)_1 v_{n-1, t} ) = v_n \in \hom_{\cC_S}((s_{n-1} \to s_n), s_0)
$$
for the degree one generating morphisms 
$$
 v_{n-1, t}\in \hom_{\cC_S}(s_{n-1}, t)
\quad v_{t, 0}\in \hom_{\cC_S}(t, s_0)
$$

\end{proof}


\section{Graded version}

In this section, we indicate how to lift our prior constructions to the traditional $\Z$-graded (rather than two-periodic) setting.

\begin{defn}
Let $\cA^{\str}$ denote the strict category with objects  $k$-linear  $A_\oo$-categories and morphisms 
$k$-linear   $A_\oo$-functors.
\end{defn}

\begin{remark}
There is the evident forgetful functor $\periodic:\cA^{\str}\to \cA^{\str}_\per$ where we tensor  up from $k$ to the graded $k$-algebra $k[u, u^{-1}]$ with $\deg u =2$.
\end{remark}

Next we introduce the natural structure needed to lift our prior constructions.
Let $\Sym^2(S^1)$ denote the second symmetric product of $S^1$. Recall that it is homeomorphic to the Moebius strip, in particular it is homotopy equivalent to a circle.

\begin{defn}
Let $\ul \Lambda_{gr}$ denote the category defined as follows.

Objects are pairs $(S, c)$  of a finite nonempty subset of the circle $S\subset S^1$ and a point $c\in \Sym^2(S^1)$.

Morphisms $(S, c)\to (S', c')$ are homotopy classes of pairs $(\varphi, \gamma)$ of a degree $1$ map $\varphi: S^1\to S^1$ such that $\varphi(S) \subset S'$ and a path $\gamma:[0,1]\to \Sym^2(S^1)$ such that $\gamma(0) = \varphi(c)$, $\gamma(1) = c'$.

Composition of morphisms $(S, c)\to (S', c') \to (S'', c'')$ is given by the homotopy class of the composition
$$
(\varphi', \gamma') \circ (\varphi, \gamma) = (\varphi'\circ \varphi, \gamma'\#_{\varphi'(c')} (\varphi'\circ \gamma))
$$ 
where we write $\#_{\varphi'(c')}$ for the concatenation of paths at $\varphi'(c') \in \Sym^2(S^1)$.
\end{defn}

\begin{remark}
Note that every morphism $(\varphi, \gamma):(S, c)\to (S', c')$ can be factored 
$$
\xymatrix{
(\varphi, \gamma) :(S, c) \ar[rr]^-{(\varphi, \gamma_{\mathit{const}})} &&(\varphi(S) =S', \varphi(c))  \ar[rr]^-{  (\id_{S^1}, \gamma) } && (S', c')
}
$$
where $\gamma_{\mathit{const}}$ denotes the constant path.
\end{remark}

\begin{remark}
There is the full and essentially surjective forgetful functor $\forget:\ul\Lambda_{gr} \to \ul\Lambda$ where we forget the points $c\in \Sym^2(S^1)$ and the paths 
 $\gamma:[0,1]\to \Sym^2(S^1)$.
 \end{remark}

Next we present a graded version of Definition~\ref{defn per cat}.

Let $S = [s_0, s_1, \ldots, s_n] \subset S^1$ be a  finite cyclically ordered nonempty set. 
As usual, we will understand the indices modulo $n+1$
so that $s_{i+n+1}$ also denotes $s_i$.

Given a point $c\in \Sym^2(S^1)$, we can assign a weight $\alpha_i \in\{0, 1, 2\}$, for all $i =0, \ldots, n$,
by taking the multiplicity of $c$ in the closed-open interval $[s_i, s_{i+1})$. Note in particular the total weight satisfies $\sum_{i=0}^{n+1} \alpha_i =2$.

\begin{defn}
Let $S = [s_0, s_1, \ldots, s_n] \subset S^1$ be a  finite cyclically ordered nonempty set. 

Let $c\in \Sym^2(S^1)$ be a point of the second symmetric power of $S^1$.

We define the unital $k$-linear $A_\oo$-category  $\cC_{S,c, gr}$ as follows.

(1) Objects: $\Ob \cC_{S, c, gr}= S$.

(2) Morphisms: the free $k$-modules generated by  the identity morphisms and  
the  additional  morphisms
$$
\xymatrix{
 v_{i} \in \hom_{\cC_S}(s_i, s_{i+1}),
&
|v_{i} |= 1 -\alpha_i,
&
\mbox{for $i =0, \ldots, n$.}
}
$$

(3)  Compositions: all are  zero except for the $k$-linear maps given by the  identity compositions
and the additional compositions
$$
\xymatrix{
\mu_{n+1}( v_{i+ n}, \ldots, v_{i}) = \id_{s_i},
&
\mbox{for $i =0, \ldots, n$.}
}
$$

We define the unital $k$-linear $A_\oo$-category  $\cC_{S, c, st, gr}$ to be the $A_\oo$-category of twisted complexes of $\cC_{S, c, gr}$.
\end{defn}

We have the following graded version of Proposition~\ref{prop quiver}. Let $\cQ_{n, gr}$ denote the  $k$-linear triangulated differential graded derived category of finite-dimensional modules over the $A_n$-quiver
$$
\xymatrix{
\underset{1}{\bullet} \ar@<+.1cm>[r] & \underset{2}{\bullet} \ar@<+.1cm>[r] & \cdots \ar@<+.1cm>[r] & \underset{n}{\bullet}
}
$$ 
For $i=1, \ldots, n$, denote by $m_i \in \cQ_n$  the free rank one module supported at the $i$th vertex and zero elsewhere.
Note that among these objects,  there are degree one generating morphisms $w_i:m_i \to m_{i+1}$, for $i=1, \ldots, n-1$, and no other linearly independent non-identity morphisms.

 The proof of the following is immediate given that of  Proposition~\ref{prop quiver}.

\begin{prop}\label{prop quiver}
Let $S= [s_0, s_1, \ldots, s_n]\subset S^1$ be a  finite cyclically ordered nonempty set.

Let $c\in \Sym^2(S^1)$ be a point of the second symmetric power of $S^1$.

Fix the element $s_0\in S$ so that the remaining elements inherit the linear order $s_1, \ldots, s_n$.

Then there is a quasi-equivalence of $A_\oo$-categories
$$
\xymatrix{
F:\cQ_{n, gr} \ar[r]^-\sim & \cC_{S, c, st, gr}
}
$$
such that $F(m_i) = s_i[\eta_i]$, for $i=1, \ldots, n$, and $F(w_i) = v_i$, for $i=1, \ldots, n-1$,
 where the shift $\eta_i \in\{0, 1, 2\}$ is the multiplicity of $c$ in the closed-open interval $[s_0, s_i)$.
 
%
\end{prop}

Now we have arrived at the graded version of Theorem~\ref{main thm}. The proof is largely a routine check of gradings given our prior constructions.

\begin{thm}\label{main thm gr}
The assignment of the $A_\oo$-category $\cC_{S,c, st, gr}$ to a pair of a finite cyclically ordered nonempty set $S \subset S^1$
and a point $c\in \Sym^2(S^1)$ naturally extends to a functor $\cC_{\mathit{st}, \mathit{gr}}:\ul\Lambda^{op}_{gr}\to \cA^\str$
fitting into a commutative diagram 
$$
\xymatrix{
\ar[d]_-\forget \ul \Lambda^{op}_{gr}\ar[r]^-{\cC_{\mathit{st}, \mathit{gr}}} &  \cA^\str \ar[d]^-\periodic \\
\ul \Lambda^{op}\ar[r]^-{\cC_{\mathit{st}}} &  \cA_\per^\str
}
$$
\end{thm}

\begin{proof}
Recall that every morphism $(\varphi, \gamma):(S, c)\to (S', c')$ can be factored 
$$
\xymatrix{
(\varphi, \gamma) :(S, c) \ar[rr]^-{(\varphi, \gamma_{\mathit{const}})} &&(\varphi(S) =S', \varphi(c))  \ar[rr]^-{  (\id_{S^1}, \gamma) } && (S', c')
}
$$
where $\gamma_{\mathit{const}}$ denotes the constant path.

Let us explain the functoriality of the assignment $\cC_{S,c, st, gr}$ under morphisms of the form  
$$
\xymatrix{
\gamma= (\id_{S^1}, \gamma) :(S, c) \ar[r] & (S, c')
}
$$ 
We will then leave the remaining compatibility check to the reader.

Given a path $\gamma:[0,1] \to \Sym^2(S^1)$, we can assign a weight $\beta_i \in\Z$, for all $i =0, \ldots, n$,
by the signed count of the number of times $\gamma(t)$ crosses $s_i \in S$ from the closed-open interval $[s_{i-1}, s_{i})$ to the   open-closed interval $[s_{i}, s_{i+1})$.

On objects, we define the induced functor to be given by shifts 
$$
\xymatrix{
\gamma^*:\cC_{S,c, st, gr}\to \cC_{S,c', st, gr} & \gamma^*s_i = s_i[-\beta_i]
}
$$
On morphisms and compositions, we set it to be the identity.

We leave the  compatibility check of our prior constructions to the reader.
\end{proof}

\begin{remark}
The  compatibility check  can be made particularly easy by the following observation: with the construction of $\gamma^*$ in hand, we can move any $c\in \Sym^2(S^1)$ to consist of two distinct points, and also move any $c\in \Sym^2(S^1)$ to arrange that a given morphism $(\varphi, \gamma):(S, c)\to (S', c')$ has the property that $\varphi$ is injective on $c$.
\end{remark}


\end{document}